\documentclass[10pt, a4paper]{amsart}

\usepackage{amscd, amsmath, amstext, amsthm, amssymb, amsopn, amsxtra, amsfonts}

\usepackage{fancyhdr, paralist, epsfig}

\usepackage[mathcal]{euscript}
\usepackage[english]{babel}
\usepackage[latin1]{inputenc}

\newtheoremstyle{normal}{5pt}{5pt}{\normalfont}{}{\bfseries}{}{0.4em}{\bfseries{\thmname{#1}\thmnumber{ #2}.\thmnote{ \hspace{0.5em}(#3)\newline}}}

\newtheoremstyle{kursiv}{5pt}{5pt}{\itshape}{}{\bfseries}{}{0.4em}{\bfseries{\thmname{#1}\thmnumber{ #2}.\thmnote{ \hspace{0.5em}(#3)\newline}}}

\theoremstyle{kursiv}
\newtheorem{thm}{Theorem}
\newtheorem{lem}[thm]{Lemma}
\newtheorem{prop}[thm]{Proposition}
\newtheorem{dfn}[thm]{Definition}
\theoremstyle{normal}
\newtheorem{ex}[thm]{Example}

\newcommand{\T}{(T(t))_{t\geqslant0}}
\newcommand{\Sg}{(S(t))_{t\geqslant0}}
\newcommand{\id}{\operatorname{id}\nolimits}
\newcommand{\rad}{\operatorname{r}\nolimits}
\renewcommand{\Re}{\operatorname{Re}\nolimits}
\newcommand{\s}{\operatorname{s}\nolimits}
\newcommand{\cs}{\operatorname{cs}\nolimits}
\newcommand{\Bigsum}[2]{\ensuremath{\mathop{\textstyle\sum}_{#1}^{#2}}}

\begin{document}

\title{The growth bound for strongly continuous\\semigroups on Fr\'{e}chet spaces}

\author{Sven-Ake Wegner\,\MakeLowercase{$^{\text{a}}$}}

\renewcommand{\thefootnote}{}
\hspace{-1000pt}\footnote{\hspace{5.5pt}2010 \emph{Mathematics Subject Classification}: Primary 47D06; Secondary 46A04, 34G10.}
\hspace{-1000pt}\footnote{\hspace{5.5pt}\emph{Key words and phrases}: semigroup, growth bound, spectral bound, power bounded operator, Fr\'{e}chet space.\vspace{1.6pt}}
\hspace{-1000pt}\footnote{$^{\text{a}}$\,Sobolev Institute of Mathematics, Pr.~Akad.~Koptyuga 4, 630090, Novosibirsk, Russia, Phone:+7\hspace{1.2pt}(8)\hspace{1.2pt}383\hspace{1.2pt}/\hspace{1.2pt}363\hspace{1.2pt}-\hspace{1.2pt}4648,\linebreak\phantom{x}\hspace{12.5pt}Fax:\hspace{1.2pt}\hspace{1.2pt}+7\hspace{1.2pt}(8)\hspace{1.2pt}383\hspace{1.2pt}/\hspace{1.2pt}333\hspace{1.2pt}-\hspace{1.2pt}2598, eMail: wegner@math.uni-wuppertal.de.}

\begin{abstract}
We introduce the concepts of growth and spectral bound for strongly continuous semigroups acting on Fr\'{e}chet spaces and show that the Banach space inequality $s(A)\leqslant\omega_0(T)$ extends to the new setting. Via a concrete example of an even uniformly continuous semigroup we illustrate that for Fr\'{e}chet spaces effects with respect to these bounds may happen that cannot occur on a Banach space.
\end{abstract}

\maketitle

%%%%%%%%%%%%%%%%%%%%%%%%%%%%%%%%%%%%%%%%%%%%%
%                                           %
%                                           %
% Introduction                              %
%                                           %
%                                           %
%%%%%%%%%%%%%%%%%%%%%%%%%%%%%%%%%%%%%%%%%%%%%

\section{Introduction}\label{SEC-1}

Not all results on strongly continuous semigroups carry over from the world of Banach spaces to those of Fr\'{e}chet spaces. For example, there exist semigroups with super exponential growth or uniformly continuous semigroups whose exponential series representations are divergent.  From another perspective, exactly these phenomena illustrate that in the Fr\'{e}chet case interesting effects may happen to take place where the picture is clear if the underlying space is Banach. We refer to Albanese, Bonet, Ricker \cite{ABR13} and Frerick, Jord\'{a}, Kalmes, Wengenroth \cite{FEJKW} for very recent results of this nature.
\smallskip
\\In this article we first show that the concept of growth bound can be transferred to Fr\'{e}chet spaces in two different ways, namely in a purely topological sense or with respect to a fixed fundamental system of seminorms. Then we consider classical and recent approaches for non Banach spectral theories to define the spectral bound. We show that for all these spectral theories and for both definitions of the growth bound the classical inequality between spectral and growth bound extends to Fr\'{e}chet spaces. Finally we consider an example of an even uniformly continuous semigroup and show that depending on the notion of growth resp.~spectral bound which we choose, their equality may fail. This effect cannot occur on a Banach space.
\smallskip
\\For the theory of Fr\'{e}chet spaces we refer to Jarchow \cite{Jarchow} and K\"othe \cite{KoetheII}. For the theory of semigroups we refer to Engel, Nagel \cite{EN}, Albanese, Bonet, Ricker \cite{ABR13}, Yosida \cite{Yosida} and K{\=o}mura \cite[Section 1]{Komura}.

%%%%%%%%%%%%%%%%%%%%%%%%%%%%%%%%%%%%%%%%%%%%%
%                                           %
%                                           %
% Growth Bound                              %
%                                           %
%                                           %
%%%%%%%%%%%%%%%%%%%%%%%%%%%%%%%%%%%%%%%%%%%%%
 
\medskip
\section{Growth Bound}\label{SEC-2}

For the whole article let $X$ be a Fr\'{e}chet space, i.e., a complete metrizable locally convex space. We denote by $\cs(X)$ the set of all continuous seminorms on $X$, by $\mathcal{B}$ the collection of all bounded subsets of $X$ and by $L(X)$ the space of all linear and continuous maps from $X$ into itself. We write $L_b(X)$, if $L(X)$ is endowed with the topology of uniform convergence on the bounded subsets of $X$ given by the seminorms $q_B(S)=\sup_{x\in B}q(Sx)$ for $S\in L(X)$, $B\in\mathcal{B}$ and $q\in\cs(X)$. By \cite[Prop.~11.2.7]{Jarchow} a subset of $L(X)$ is equicontinuous if and only if it is bounded in $L_b(X)$ and by \cite[\S 39, 6(5)]{KoetheII}, $L_b(X)$ is a complete locally convex space.
\smallskip
\\Also for the whole article let $\T$ be a $C_0$-semigroup on $X$. That is, $T(t)\in L(X)$ for $t\geqslant0$, $T(0)=\id_X$, $T(t+s)=T(t)T(s)$ for $t$, $s\geqslant0$ and $\lim_{t\rightarrow t_0}T(t)x=T(t_0)x$ for $x\in X$ and $t_0\geqslant0$. We say that $\T$ is exponentially equicontinuous of order $\omega\in\mathbb{R}$, if
\begin{equation*}
\forall\:q\in\Gamma\;\exists\:p\in\Gamma,\,M\geqslant1\;\forall\:t\geqslant0,\,x\in X\colon q(T(t)x)\leqslant M e^{\omega t} p(x)
\end{equation*}
holds for some or, equivalently, for every fundamental system of seminorms $\Gamma\subseteq\cs(X)$. The $C_0$-semigroup is said to be exponentially equicontinuous, if the above condition holds for some $\omega$ and equicontinuous, if the latter holds with $\omega=0$. Note that $\T$ is exponentially equicontinuous of order $\omega$ if and only if $\{e^{-\omega t}T(t)\}_{t\geqslant0}$ is equicontinuous or, equivalently, bounded in $L_b(X)$.

\begin{dfn} Let $\T$ be an exponentially equicontinuous $C_0$-semigroup. We denote by $\omega_0(T)$ the infimum over all $\omega\in\mathbb{R}$ for which 
$\{e^{-\omega t}T(t)\}_{t\geqslant0}$ is equicontinuous and call this number the \emph{growth bound} of $\T$.
\end{dfn}

The main ingredient for the following Lemma \ref{LEM-0} is the \textquotedblleft{}recalibration trick\textquotedblright{} of Moore \cite[Thm.~4]{Moore}. According to Joseph \cite[Dfn.~3.1]{J} we call $T\in L(X)$ universally bounded with respect to a fundamental system $\Gamma$ if there exists $N\geqslant0$ such that $q(Tx)\leqslant Nq(x)$ holds for all $q\in\Gamma$ and $x\in X$.

\begin{lem}\label{LEM-0} Let $\T$ be a $C_0$-semigroup and $\omega\in\mathbb{R}$. The following are equivalent.\vspace{2pt}
\begin{compactitem}
\item[(i)] The semigroup $\T$ is exponentially equicontinuous of order $\omega$.\vspace{2pt}
\item[(ii)] There exists a fundamental system $\Gamma$ and a constant $M\geqslant1$ such that 
\begin{equation*}
\forall\:q\in\Gamma,\,t\geqslant0,\,x\in X \colon q(T(t)x)\leqslant Me^{\omega t}q(x)
\end{equation*}
holds.\vspace{2pt}
\item[(iii)] There exists a fundamental system $\Gamma$ such that every operator in the semigroup is universally bounded with respect to $\Gamma$ and there exists a constant $M\geqslant1$ such that 
\begin{equation*}
\|T(t)\|_{\Gamma}=\sup_{q\in\Gamma}\sup_{q(x)\leqslant1}q(T(t)x)\leqslant M e^{\omega t}
\end{equation*}
holds for all $t\geqslant0$.
\end{compactitem}
\end{lem}
\begin{proof}\textquotedblleft{}(i)$\Rightarrow$(ii)\textquotedblright{} Let $\Gamma'$ be some fundamental system and let $\omega$ be such that $\{e^{-\omega t}T(t)\}_{t\geqslant0}$ is equicontinuous. Define $\Gamma=\{q\:;\:q'\in\Gamma'\}$ via $q(x)=\sup_{t\geqslant0}q'(e^{-\omega t}T(t)x)$ for $x\in X$. By \cite[Rem.~2.2(i)]{ABR13}, $\Gamma$ is a fundamental system for $X$ satisfying the desired property with $M=1$. The converse is trivial.
\medskip
\\\textquotedblleft{}(iii)$\Rightarrow$(ii)\textquotedblright{} Let $\Gamma$ be as in (iii). For $t\geqslant0$ there exists $N$ such that $q(T(t)x)\leqslant Nq(x)$ holds for all $q\in\Gamma$ and $x\in X$. We note that therefore, $q(x)=0$ always implies $q(T(t)x)=0$. Now we fix $q\in\Gamma$ and select $M\geqslant1$ according to (iii). Let $t\geqslant0$ and $x\in X$ be given. If $q(x)=0$ it follows $q(T(t)x)=0$ by the first sentence of this paragraph. Otherwise, we can select $\lambda>0$ such that $q(\lambda x)=1$. Then, $\lambda q(T(t)x)=q(T(t)(\lambda x))\leqslant M e^{\omega t}= M e^{\omega t}q(\lambda x)=\lambda M e^{\omega t}q(x)$ holds and whence $q(T(t)x)\leqslant M e^{\omega t}q(x)$ is valid. Again, the converse is trivial.
\end{proof}

Let us mention that $\T$ can be exponentially equicontinuous but for some $\Gamma$ the condition in Lemma \ref{LEM-0}(ii) can fail for every $\omega$. Consider $X=\mathbb{C}^2$, $T(t)x=(x_1+tx_2, x_2)$ for $x=(x_1,x_2)$, $t\geqslant0$. Then $\omega_0(T)=0$,  cf.~\cite[Ex.~I.5.7(i)]{EN}. Let $\omega>0$ and $\Gamma=\{p_1,p_2\}$, $p_1(x)=|x_1|$, $p_2(x)=\max_{i=1,2}|x_i|$. Let $M\geqslant1$, select $t>0$ and $x=(0,1)$. Then, $p_1(T(t)x)=t>0$ and $p_1(x)=0$ hold. It follows $\|T(t)\|_{\Gamma}=\infty$ for $t>0$.
\medskip
\\From Lemma \ref{LEM-0} we get
\begin{equation*}
\omega_0(T)=\inf\{\omega\in\mathbb{R}\:;\:\exists\:\Gamma,\,M\geqslant1\;\forall\:q\in\Gamma,\,t\geqslant0,\,x\in X\colon q(T(t)x)\leqslant Me^{\omega t}q(x)\}
\end{equation*}
where we may also restrict ourselves to $M=1$ in view of the proof of Lemma \ref{LEM-0}. If one allows $M\geqslant1$, in the Banach space analogue of the above formula \cite[Dfn.~I.5.6]{EN} it is not necessary to take the infimum over all norms which induce the topology of the underlying space. For Fr\'{e}chet spaces and fundamental systems this is not the case.

\begin{dfn} Let $\T$ be an exponentially equicontinuous $C_0$-semigroup. We define
\begin{equation*}
\omega_{0,\Gamma}(T)=\inf\{\omega\in\mathbb{R}\:;\:\exists\:M\geqslant1\;\forall\:q\in\Gamma,\,t\geqslant0,\,x\in X\colon q(T(t)x)\leqslant Me^{\omega t}q(x)\}
\end{equation*}
for a fixed fundamental system $\Gamma$. By definition, $\omega_0(T)$ is the infimum over all the $\omega_{0,\Gamma}(T)$.
\end{dfn}

Using the fact that the proof of \textquotedblleft{}(ii)$\Leftrightarrow$(iii)\textquotedblright{} in Lemma \ref{LEM-0} worked without a recalibration, variants of Banach space growth bound formulas \cite[Section IV.2]{EN} for $\omega_{0,\Gamma}(T)$ can be established by following the lines of their classical proofs.

\begin{prop}\label{PROP-1} Let $\T$ be an exponentially equicontinuous $C_0$-semigroup and $\Gamma$ be a fundamental system which satisfies the equivalent conditions in Lemma \ref{LEM-0}(ii) and Lemma \ref{LEM-0}(iii). Then,
\begin{equation*}
\omega_{0,\Gamma}(T)=\inf_{t>0}{\textstyle\frac{1}{t}}\log\|T(t)\|_{\Gamma} = \lim_{t\rightarrow\infty}{\textstyle\frac{1}{t}}\log\|T(t)\|_{\Gamma}={\textstyle\frac{1}{t_0}}\log(\lim_{n\rightarrow\infty}\|T(t_0)^n\|_{\Gamma}^{1/n})
\end{equation*}
holds for arbitrary $t_0>0$.
\end{prop}
\begin{proof} We start with the first and the second equality and abbreviate $\omega_{0,\Gamma}=\omega_{0,\Gamma}(T)$. Define $\xi\colon[0,\infty)\rightarrow\mathbb{R}$, $\xi(t)=\log\|T(t)\|_{\Gamma}$, which is bounded on compact intervals. Fix $s$, $t\geqslant0$ and $q\in\Gamma$. Put $C=\sup_{q(y)\leqslant1}q(T(t)y)<\infty$. If $q(x)\not=0$ we select $\lambda>0$ such that $q(\lambda x)=1$ and compute $q(T(t)x)=\lambda^{-1}q(T(t)(\lambda x))\leqslant\lambda^{-1}C q(\lambda x)=C q(x)$. If $q(x)=0$, it follows $q(T(t)x)=0$, cf.~the proof of Lemma \ref{LEM-0}. Therefore the estimate $q(T(t)x)\leqslant Cq(x)$ is true for all $x\in X$. We thus get that $q(T(t)T(s)x)\leqslant \sup_{q(y)\leqslant1}q(T(t)y)\cdot q(T(s)x)$ holds for every $x\in X$ and obtain finally the estimate $\sup_{q(x)\leqslant1}q(T(s)T(t)x)\leqslant\sup_{q(x)\leqslant1}q(T(s)x)\cdot\sup_{q(x)\leqslant1}q(T(t)x)$. From the latter it follows that $\xi$ is subadditive. From \cite[Lem.~IV.2.3]{EN} we obtain that $\omega'=\inf_{t\geqslant0}{\
textstyle \frac{1}{t}}\log\|T(t)\|_{\Gamma}=\lim_{t\rightarrow\infty}{\textstyle\frac{1}{t}}\log\|T(t)\|_{\Gamma}$ exists in $\mathbb{R}$. We deduce that $e^{
\omega't}\leqslant\|T(t)\|_{\Gamma}$ is valid for all $t\geqslant0$. Hence, $\omega'\leqslant\omega_{0,\Gamma}$ must hold. Let $\omega>\omega'$. Then there exists $t_0>0$ such that $\frac{1}{t}\log\|T(t)\|_{\Gamma}\leqslant\omega$, hence $\|T(t)\|_{\Gamma}\leqslant e^{\omega t}$, holds for all $t\geqslant t_0$. Since $\sup_{0\leqslant{}t\leqslant{}t_0}\|T(t)\|_{\Gamma}<\infty$, we find $M\geqslant1$ such that $\|T(t)\|_{\Gamma}\leqslant Me^{\omega t}$ holds for every $t\geqslant0$ and thus $\omega_{0,\Gamma}\leqslant\omega$ must hold. Since $\omega>\omega'$ was arbitrary, $\omega_{0,\Gamma}\leqslant\omega'$ follows. To finish the proof, we compute $\omega_{0,\Gamma}=\lim_{n\rightarrow\infty}({\textstyle\frac{1}{nt_0}}\log\|T(nt_0)\|_{\Gamma})={\textstyle\frac{1}{t_0}}\lim_{n\rightarrow\infty}(\log\|T(t_0)^n\|_{\Gamma}^{1/n})$.
\end{proof}

If $X$ is a Banach space, $\omega_0(T)$ coincides with the classical growth bound; in contrast $\omega_{0,\Gamma}(T)$ must not even be finite. On the other hand a \textquotedblleft{}reasonable\textquotedblright{} choice for $\Gamma$, e.g., $\Gamma=\{\|\cdot\|\}$ for any norm $\|\cdot\|$ inducing the topology of $X$, yields $\omega_0(T)=\omega_{0,\Gamma}(T)$ whenever $X$ is Banach. For Fr\'{e}chet spaces, the equality $\omega_0(T)=\omega_{0,\Gamma}(T)$ can fail even for \textquotedblleft{}nice\textquotedblright{} fundamental systems $\Gamma$, cf.~Example \ref{EX-0}.

%%%%%%%%%%%%%%%%%%%%%%%%%%%%%%%%%%%%%%%%%%%%%
%                                           %
%                                           %
% Spectral bound                            %
%                                           %
%                                           %
%%%%%%%%%%%%%%%%%%%%%%%%%%%%%%%%%%%%%%%%%%%%%

\medskip
\section{Spectral Bound}\label{SEC-3}

The space $L_b(X)$ is an associative algebra with identity in which left resp.~right multiplication with a fixed element defines a continuous map. Using \cite[Prop.~2.6]{Allan} it follows that $L_b(X)$ is a pseudo complete locally convex algebra in the sense of Allan \cite[Terminology~1.1 and Dfn.~2.5]{Allan}. According to \cite[Dfn.~2.1]{Allan} we say that $B\in L_b(X)$ is bounded if there exists $\mu\in\mathbb{C}\backslash\{0\}$ such that $\{(\mu{}B)^n\}_{n\in{\mathbb{N}_0}}\subseteq L_b(X)$ is bounded and denote the set of bounded elements by $L_b(X)_0$. We use the convention $\mathbb{N}=\{1,2,\dots\}$ and $\mathbb{N}_0=\{0,1,2,\dots\}$.

\begin{dfn}
Let $D(A)\subseteq X$ be a linear subspace and $A\colon D(A)\rightarrow X$ be linear. We put $\delta_A=\{\infty\}$, if $A\in L_b(X)_0$, and $\delta_A=\varnothing$ otherwise. Then the resolvent set of $A$ is defined by
\begin{equation*}
\rho(A) = \big\{ \lambda\in\mathbb{C}\:;\:\lambda-A\colon D(A)\rightarrow X \text{ is bijective and }R(\lambda,A)=(\lambda-A)^{-1}\in L_b(X)_0 \big\} \cup \delta_A.
\end{equation*}
The spectrum is defined via $\sigma(A)=\overline{\mathbb{C}}\,\backslash\,\rho(A)$ and the spectral bound by $\s(A)=\sup\{\Re\lambda\:;\:\lambda\in\sigma(A)\cap\mathbb{C}\}$.
\end{dfn}

For $A\in L(X)$ the above definition of resolvent and spectrum coincides with the those given by Allan in \cite[Dfn.~3.1 and Dfn.~3.2]{Allan}. We need the following characterization of the complex part of the resolvent.

\begin{lem}\label{LEM-2} Let $A\colon D(A)\rightarrow X$ be a linear operator and $\mathcal{U}\subseteq\mathbb{C}$. Assume that $R(\lambda,A)$ exists for every $\lambda\in\mathcal{U}$. Then, $\mathcal{U}\subseteq\rho(A)$ holds if and only if for every $\lambda\in\mathcal{U}$ there exists a fundamental system $\Gamma$ and a constant $M_{\lambda}>0$ such that $p(R(\lambda,A)x)\leqslant{}M_{\lambda}\,p(x)$ is valid for all $p\in\Gamma$ and $x\in X$.
\end{lem}
\begin{proof}\textquotedblleft{}$\Rightarrow$\textquotedblright{} Let $\Gamma$ be some fundamental system. For $\lambda\in\mathcal{U}\subseteq\rho(A)$ we select $\mu\in\mathbb{C}\backslash\{0\}$ such that $\{[\mu{}R(\lambda,A)]^n\}_{n\in{\mathbb{N}_0}}$ is equicontinuous. We define $\Gamma'=\{p'\:;\:p\in\Gamma\}$ via $p'(x)=\sup_{n\in{\mathbb{N}_0}}p([\mu{}R(\lambda,A)]^nx)$ for $x\in X$. For every $p\in\Gamma$ we find $q\in\Gamma$ and $M\geqslant0$ such that $p([\mu{}R(\lambda,A)]^nx)\leqslant{}M q(x)$ holds for all $n$ and $x$ and thus $p'\leqslant Mq$ is valid. On the other hand for every $p\in\Gamma$ we have $p\leqslant{}p'$. It follows that $\Gamma'$ is a fundamental system for $X$. We put $M_{\lambda}=1/\mu$ and compute $p'(R(\lambda,A)x)=\sup_{n\in{\mathbb{N}_0}}p([\mu{}R(\lambda,A)]^{n}R(\lambda,A)x)=\mu^{-1}\sup_{n\in{\mathbb{N}_0}}p([\mu{}R(\lambda,A)]^{n+1}x)\leqslant M_{\lambda}\,p'(x)$ for arbitrary $p'$ and $x$.
\smallskip
\\\textquotedblleft{}$\Leftarrow$\textquotedblright{} For $\lambda\in\mathcal{U}$ we select $\Gamma$ and $M_{\lambda}$ as in the condition and put $\mu=M_{\lambda}^{-1}$. For $p\in\Gamma$ and $n\in{\mathbb{N}_0}$ we compute $p([\mu{}R(\lambda,A)]^nx)\leqslant{}M_{\lambda}^{-n}p(R(\lambda,A)^nx)\leqslant{}M_{\lambda}^{-n}M_{\lambda}^n\,p(x)=p(x)$, i.e., $R(\lambda,A)\in L_b(X)_0$.
\end{proof}

Recently, Albanese, Bonet, Ricker \cite[Section 3]{ABR13} defined resolvent and spectrum for non-continuous operators on locally convex spaces in a different way; in their notation $\lambda\in\mathbb{C}$ is in the resolvent if and only if $R(\lambda,A)$ exists in $L(X)$ and thus their spectrum is a subset of $\sigma(A)\cap\mathbb{C}$. In \cite[Equation (3.7)]{ABR13}, they studied the condition of Lemma \ref{LEM-2}
as an additional property of points in the resolvent.
\smallskip
\\Arikan, Runov, Zahariuta \cite{ARZ} introduced the ultraspectrum for operators in $A\in L_b(X)$. In their notation, $\lambda\in\mathbb{C}$ is a strictly regular point of $A$ if $R(\lambda,A)$ exists in $L(X)$ and is tamable \cite[p.~29]{ARZ}, i.e., there exists a fundamental system of seminorms $\Gamma$ for $X$ such that for all $q\in\Gamma$ there exists $C\geqslant0$ such that $q(R(\lambda,A)x)\leqslant{}C q(x)$ holds for all $x\in X$. The point $\lambda=\infty$ is strictly regular if $A$ itself is tamable. The ultraspectrum is the complement of the strictly regular points in $\overline{\mathbb{C}}$. See \cite[Dfn.~10 and Thm.~14]{ARZ} for details. By Lemma \ref{LEM-2} it follows that the ultraspectrum is contained in $\sigma(A)$; for $\lambda=\infty$ an argument similar to those in the proof of Lemma \ref{LEM-2} can be used.
\smallskip
\\If $A$ is the generator of the $C_0$-semigroup $\T$, i.e.,
\begin{equation*}
Ax=\lim_{t\searrow0}{\textstyle\frac{T(t)x-x}{t}} \;\text{ for }\; x\in D(A)=\{x\in X\:;\:\lim_{t\searrow0}{\textstyle\frac{T(t)x-x}{t}}\text{ exists}\},
\end{equation*}
the spectral theory of Allan yields the following relation between the spectral bound of the generator and the semigroup itself if $A$ is a so-called power bounded operator. According to Allan, we put $\rad(B)=\sup\{|\lambda|\:;\:\lambda\in\sigma(B)\}$ for $B\in L(X)$ and $|\infty|=+\infty$.

\begin{prop}\label{PROP-2} Let $\T$ be an exponentially equicontinuous $C_0$-semigroup. Let $(A,D(A))$ be its generator and assume that $\{A^n\}_{n\in{\mathbb{N}_0}}\subseteq L(X)$ is equicontinuous. Then $s(A)=\log\rad(T(1))$ holds.
\end{prop}
\begin{proof} By our assumption we have $A\in L_b(X)_0$ and thus $\infty\not\in\sigma(A)$. By Allan's spectral mapping theorem, see \cite[Prop.~6.11]{Allan} and \cite[p.~414]{Allan} for the definition of those class of functions for which Allan's functional calculus is designed, we have $\sigma(e^{A})=e^{\sigma(A)}$. By \cite[Section IX.6]{Yosida} we have $T(1)x=\Bigsum{n=0}{\infty}{\textstyle\frac{1}{n!}}A^nx=e^{A}x$ for every $x\in X$. We compute $e^{\s(A)}=\sup\{e^{\Re\lambda}\:;\:\lambda\in\sigma(A)\}=\sup\{|\lambda|\:;\:\lambda\in\sigma(e^{A})\}=\rad(T(1))$ and obtain $\s(A)=\log\rad(T(1))$.
\end{proof}

If $X$ is Banach, the above formula is one of the ingredients to prove that $s(A)=\omega_0(T)$ holds for every uniformly continuous semigroup $\T$, cf.~the comments at the end of Section \ref{SEC-5}.

%%%%%%%%%%%%%%%%%%%%%%%%%%%%%%%%%%%%%%%%%%%%%
%                                           %
%                                           %
% Inequality                                %
%                                           %
%                                           %
%%%%%%%%%%%%%%%%%%%%%%%%%%%%%%%%%%%%%%%%%%%%%

\medskip
\section{Inequality}\label{SEC-4}

We have the following general inequality between spectral bound and growth bound, cf.~\cite[Prop.~II.2.2]{EN}.

\begin{thm}\label{THM} Let $\T$ be an exponentially equicontinuous $C_0$-semigroup with generator $(A,D(A))$. Then we have $-\infty\leqslant\s(A)\leqslant\omega_0(T)<+\infty$.
\end{thm}
\begin{proof} By rescaling we may assume w.l.o.g.~that $\omega_0=\omega_0(T)=0$. Then, $(e^{-\omega{}t}T(t))_{t\geqslant0}\subseteq L_b(X)$ is equicontinuous for every $\omega>0$. With \cite[Lem.~4.4]{ABR13} it follows that $R(\lambda,A)$ exists in $L(X)$ for every $\lambda\in\mathbb{C}_+=\{\lambda\in\mathbb{C}\:;\:\Re\lambda>0\}$. We claim that for every $\lambda\in\mathbb{C}_+$ there exists $\mu\in\mathbb{C}\backslash\{0\}$ such that $\{[\mu{}R(\lambda,A)]^n\}_{n\in{\mathbb{N}_0}}\subseteq L(X)$ is equicontinuous.
\smallskip
\\We fix $\omega>0$ such that $\Sg$, $S(t)=e^{-\omega{}t}T(t)$, $t\geqslant0$, is equicontinuous. The generator of $\Sg$ is $(B,D(B))=(A-\omega,D(A))$. By \cite[last paragraph of Rem.~3.5(iv)]{ABR13} we have
\begin{equation*}
\forall\:\lambda'\in\mathbb{C}_+\;\exists\:M_{\lambda'}>0\;\forall\:p\in\Gamma,\,x\in X\colon p(R(\lambda',B)x)\leqslant{}M_{\lambda'}p(x).
\end{equation*}
From \cite[proof of Lem.~4.4]{ABR13} we get that $R(\lambda,A)=R(\lambda-\omega,B)$ holds for every $\lambda\in\mathbb{C}$ with $\Re\lambda>\omega$.
\smallskip
\\Let now $\lambda\in\mathbb{C}_+$ be given. Select $0<\omega<\Re\lambda$ and put $\lambda'=\lambda-\omega\in\mathbb{C}_+$. We select $M_{\lambda'}>0$ as above and put $\mu=1/M_{\lambda'}$. For given $p\in\Gamma$, $n\in{\mathbb{N}_0}$ and $x\in X$ we have by iteration $p(R(\lambda',B)^nx)\leqslant M_{\lambda'}^np(x)$ and thus
$p([\mu{}R(\lambda,A)]^nx)=p([M_{\lambda'}^{-1}R(\lambda-\omega,B)]^nx)=M_{\lambda'}^{-n}p([R(\lambda',B)]^nx)\leqslant p(x)$ which establishes the claim. Consequently, we have $\mathbb{C}_+\subseteq\rho(A)$, i.e., every $\lambda\in\sigma(A)\cap\mathbb{C}$ satisfies $\Re\lambda\leqslant0$. Therefore, $\s(A)\leqslant0$ follows.
\end{proof}

The inequality in Theorem \ref{THM} remains valid if we replace our notion of spectrum by those of \cite{ABR13}. If $A$ belongs to $L(X)$ and we replace $\sigma(A)$ by the ultraspectrum \cite{ARZ} the estimate also applies.

%%%%%%%%%%%%%%%%%%%%%%%%%%%%%%%%%%%%%%%%%%%%%
%                                           %
%                                           %
% Example                                   %
%                                           %
%                                           %
%%%%%%%%%%%%%%%%%%%%%%%%%%%%%%%%%%%%%%%%%%%%%

\medskip
\section{Example}\label{SEC-5}

The following example was studied by Albanese, Bonet, Ricker in \cite[Rem.~3.5(v)]{ABR13}. Below, we correct some flaws and extend their investigation.

\begin{ex}\label{EX-0} Let $X=\mathbb{C}^{\mathbb{N}}$ be the space of all complex sequences endowed with the topology of coordinate wise convergence given by $\Gamma=\{p_n\:;\:n\in\mathbb{N}\}$, $p_n(x)=\max_{j=1,\dots,n}|x_j|$ for $x=(x_1,x_2,\dots)$. The operator $A\colon X\rightarrow X$, $Ax=(0,x_1,x_2,\dots)$, denotes the right shift on $X$. We have $\{A^n\}_{n\in{\mathbb{N}_0}}\subseteq L(X)$ and the latter is an equicontinuous set. The strongly continuous semigroup $\T$ generated by $A\colon X\rightarrow X$ is thus given by the power series
\begin{equation*}
T(t)x=\Bigsum{k=0}{\infty}{\textstyle\frac{1}{k!}}(tA)^kx=(x_1,x_2+tx_1,x_3+tx_2+{\textstyle\frac{t^2}{2!}}x_1,x_4+tx_3+{\textstyle\frac{t^2}{2!}}x_2+{\textstyle\frac{t^3}{3!}}x_1,\cdots)
\end{equation*}
for $x\in X$ and $t\geqslant0$, see \cite[p.~245]{Yosida} and \cite[Rem.~3.5(v)]{ABR13}. We observe that for given $\omega>0$ and $n\in\mathbb{N}$ there exists $M\geqslant 1$ such that $p_n(T(t)x)\leqslant Me^{\omega t}p_n(x)$ holds for all $t\geqslant0$ and $x\in X$. Whence, $\omega_0(T)$ must be less or equal to zero. On the other hand we compute $p_n(T(t)x)=1+t+\cdots+t^{n-1}/(n-1)!$ for $x=(1,1,\dots)\in X$. This shows that in the previous condition for $0<\omega<1$ the constant $M\geqslant1$ cannot be independent of $n\in\mathbb{N}$. Consequently, $\omega_{0,\Gamma}(T)=1$.
\smallskip
\\For $\lambda=0$ the map $\lambda-A$ is not surjective, for $\lambda\not=0$ it can be checked that $R(\lambda,A)$ exists in $L(X)$ and is given by the formula
\begin{equation*}
R(\lambda,A)x=({\textstyle\frac{1}{\lambda}}x_1,{\textstyle\frac{1}{\lambda}}x_2+{\textstyle\frac{1}{\lambda^2}}x_1,{\textstyle\frac{1}{\lambda}}x_3+{\textstyle\frac{1}{\lambda^2}}x_2+{\textstyle\frac{1}{\lambda^3}}x_1,\dots)
\end{equation*}
for $x\in X$. We estimate
\begin{eqnarray*}
p_n(R(\lambda,A)x)
& \leqslant & \max\{{\textstyle\frac{1}{|\lambda|}}|x_1|,{\textstyle\frac{1}{|\lambda|}}|x_2|+{\textstyle\frac{1}{|\lambda|^2}}|x_1|,\dots,{\textstyle\frac{1}{|\lambda|}}|x_n|+\cdots+{\textstyle\frac{1}{|\lambda|^n}}|x_1|\}\\
&\leqslant& \max\{{\textstyle\frac{1}{|\lambda|}},{\textstyle\frac{1}{|\lambda|}}+{\textstyle\frac{1}{|\lambda|^2}},\dots,{\textstyle\frac{1}{|\lambda|}}+\cdots+{\textstyle\frac{1}{|\lambda|^n}}\}\,\max\{|x_1|,\dots,|x_n|\}\\
& \leqslant & \Bigsum{j=1}{n}({\textstyle\frac{1}{|\lambda|}})^{j}\,p_n(x)\;=\;{\textstyle\frac{1-1/|\lambda|^n}{|\lambda|-1}}\,p_n(x)\;\leqslant\; {\textstyle\frac{1}{|\lambda|-1}}\,p_n(x)
\end{eqnarray*}
for $|\lambda|>1$, $x\in X$ and $n\in\mathbb{N}$. Using Lemma \ref{LEM-2} with $\mathcal{U}=\{z\in\mathbb{C}\:;\:|z|>1\}$ and $M_\lambda=(|\lambda|-1)^{-1}$ for $\lambda\in\mathcal{U}$ it follows that all $\lambda\in\mathbb{C}$ with $|\lambda|>1$ belong to $\rho(A)$. For $\mathcal{U}'=\{z\in\mathbb{C}\:;\:0<|z|\leqslant1\}$ the condition
\begin{equation*}
\forall\:\lambda\in\mathcal{U}'\;\exists\:M_{\lambda}>0\;\forall\:p\in\Gamma,\,x\in X\colon p(R(\lambda,A)x)\leqslant M_{\lambda}p(x)
\end{equation*}
is not satisfied. Let $|\lambda|\leqslant1$ with $\lambda\not=1$. For $x=(1,1,\dots)$ we get
\begin{equation*}
p_n(R(\lambda,A)x)=\max\{|{\textstyle\frac{1}{\lambda}}|,|{\textstyle\frac{1}{\lambda}}+{\textstyle\frac{1}{\lambda^2}}|,\dots,|{\textstyle\frac{1}{\lambda}}+\cdots+{\textstyle\frac{1}{\lambda^n}}|\}=\max_{1\leqslant{}k\leqslant{}n}\big|{\textstyle\frac{1-1/\lambda^k}{\lambda-1}}\big|\geqslant \big|{\textstyle\frac{1-1/\lambda^n}{\lambda-1}}\big|\stackrel{n\rightarrow\infty}{\longrightarrow}\infty.
\end{equation*}
For $\lambda=1$ and $x$ as above we get $p_n(R(\lambda,A)x)=n$. Therefore, for no $\lambda\in\mathbb{C}$ with $0<|\lambda|\leqslant1$ there can exist $M_{\lambda}$ such that $p_n(R(\lambda,A)x)\leqslant M_{\lambda}\,p_n(x)$ holds for every $x\in X$ and $n\in\mathbb{N}$. Nevertheless, $\mathcal{U}'\subseteq\rho(A)$ holds. In particular, the aforementioned inequality becomes true if we switch to another system of seminorms, see Lemma \ref{LEM-2}. We fix $x\in X$ and $\lambda\in\mathcal{U}'$. By induction we show that
\begin{equation*}
[n]\;\;\;\;\;\;\forall\:k\in\mathbb{N}\;\colon\;(R(\lambda,A)^nx)_k =\Bigsum{j=1}{k}\lambda^{-(n+j-1)}\,{\textstyle\binom{n-2+j}{j-1}}\,x_{k-j+1}
\end{equation*}
holds for every $n\in\mathbb{N}$. We observe that we established the statement for $n=1$ already above. Moreover, we have the recursion formula $(R(\lambda,A)x)_k=\lambda^{-1}x_k+\lambda^{-1}(R(\lambda,A)x)_{k-1}$ for arbitrary $k$ if we put $(R(\lambda,A)x)_{0}=0$. Now we show $[n]\Rightarrow[n+1]$. In order to establish $[n+1]$ we proceed by induction over $k$. Using the recursion formula and then $[n]$ we get $(R(\lambda,A)^{n+1}x)_1=\lambda^{-1}(R(\lambda,A)^nx)_1=\lambda^{-(n+1)}x_1$ which coincides with the right hand side of the equation in $[n+1]$ for $k=1$. For $k\geqslant2$ we use the recursion formula to get $(R(\lambda,A)^{n+1}x)_k=\lambda^{-1}(R(\lambda,A)^nx)_k+\lambda^{-1} (R(\lambda,A)^{n+1}x)_{k-1}$. For the first summand we use the induction hypothesis of the induction over $n$ and for the second those of the induction over $k$. Then we obtain
\begin{eqnarray*}
(R(\lambda,A)^{n+1}x)_k & = & \lambda^{-1}\Bigsum{j=1}{k}\lambda^{-(n+j-1)}\,{\textstyle\binom{n-2+j}{j-1}}\,x_{k-j+1}+\lambda^{-1}\Bigsum{j=1}{k-1}\lambda^{-(n+j+1)}\,{\textstyle\binom{n-1+j}{j-1}}\,x_{k+j}\\
& = & \lambda^{-(n+1)}\,x_{k}+\Bigsum{j=2}{k}\lambda^{-(n+j)}\,\big[{\textstyle\binom{n-2+j}{j-1}}+{\textstyle\binom{n-2+j}{j-2}}\big]\,x_{k-j+1}\\
& = & \Bigsum{j=1}{k}\lambda^{-(n+j)}\,{\textstyle\binom{n-1+j}{j-1}}\,x_{k-j+1}
\end{eqnarray*}
which is the right hand side of the equation in $[n+1]$ for $k$. This finishes both inductions. We claim
\begin{equation*}
\forall\:\lambda\in\mathcal{U}'\;\exists\:\mu\in\mathbb{C}\backslash\{0\}\;\forall\:m\in\mathbb{N}\;\exists\:K\geqslant0\;\forall\:n\in{\mathbb{N}_0},\,x\in X \colon p_{m}([\mu{}R(\lambda,A)]^nx)\leqslant K p_{m}(x)
\end{equation*}
which implies that $R(\lambda,A)\in L_b(X)_0$ holds for every $\lambda\in\mathcal{U}'$. For given $\lambda\in\mathcal{U}'$ select $\mu$ such that $|\mu|/|\lambda|\leqslant1/2$ holds. Let $m$ be given. Put $K=|\lambda|^{-(m-1)}e^{m}\sup_{n\in\mathbb{N}}2^{-n}(n+1)^{m}$. Let $n$ and $x$ be given. Since $K\geqslant1$ we are done if $n=0$. Otherwise we estimate
\begin{eqnarray*}
p_{m}([\mu{}R(\lambda,A)]^nx) & \leqslant & |\mu|^{n} \max_{1\leqslant{}k\leqslant{}m}\,\Bigsum{j=1}{k}|\lambda|^{-(n+j-1)}{\textstyle\binom{n-2+j}{j-1}}|x_{k-j+1}|\\
&\leqslant & \big({\textstyle\frac{|\mu|}{|\lambda|}}\big)^{n}\max_{1\leqslant{}k\leqslant{}m}\,\Bigsum{j=1}{k}|\lambda|^{-(j-1)}{\textstyle\binom{n-2+j}{j-1}}\,p_k(x)\\
&\leqslant & 2^{-n}\,|\lambda|^{-(m-1)}\,\big(\max_{1\leqslant{}k\leqslant{}m}\,\Bigsum{j=1}{k}{\textstyle\binom{n-2+j}{j-1}}\big)\,p_{m}(x)\;\leqslant\;K\,p_{m}(x)
\end{eqnarray*}
where the last inequality holds since
\begin{equation*}
\Bigsum{j=1}{k}{\textstyle\binom{n-2+j}{j-1}} \leqslant \Bigsum{j=0}{k}{\textstyle\binom{n-1+j}{j}}={\textstyle\binom{n+k}{k}}\leqslant \big({\textstyle\frac{(n+k)e}{k}}\big)^{k}\leqslant e^k ({\textstyle\frac{n}{k}}+1)^k\leqslant{}e^{m}(n+1)^{m}
\end{equation*}
is true whenever $1\leqslant{}k\leqslant{}m$. Finally, we get $\sigma(A)=\{0\}$ and thus $\s(A)=0$.
\end{ex}

If $X$ is a Banach space then $\s(A)=\omega_0(T)$ holds whenever $\T$ is uniformly continuous. The proof for this equality employs Hadamard's formula for the spectral radius $\rad(\cdot)$ and the Banach space versions of Proposition \ref{PROP-1} and Proposition \ref{PROP-2}. Indeed, Allan \cite[Thm.~3.12 and Prop.~2.18]{Allan} proved a Hadamard type formula for the spectral radius within his theory. But neither the formula $\rad(T(t_0))=\lim_{n\rightarrow\infty}\|T(t_0)^n\|_{\Gamma}^{1/n}$ nor the equality $\omega_{0,\Gamma}(T)=\frac{1}{t_0}\log\rad(T(t_0))$ extend in general to the Fr\'{e}chet case---not even for \textquotedblleft{}nice\textquotedblright{} fundamental systems $\Gamma$. Already for $t_0=1$ the above and the Propositions \ref{PROP-1} and \ref{PROP-2} would imply that $\s(A)=\omega_{0,\Gamma}(T)$ holds in Example \ref{EX-0}. On the other hand, it seems to be open if $s(A)=\omega_0(T)$ holds in general or can fail for uniformly continuous semigroups on Fr\'{e}chet spaces when $s(A)$ is defined 
as in Section \ref{SEC-3}, i.e., with respect to the spectral theory of \cite{Allan}, or with respect to the ultraspectrum \cite{ARZ}.

\bigskip

\footnotesize

{\sc Acknowledgements. }The author's stay at the Sobolov Institute of Mathematics at Novosibirsk was supported by the German Academic Exchange Service (DAAD). The author likes to express his gratitude to the referee for pointing out a mistake in the original version of this paper and for giving several valuable comments and references.

\normalsize

\bibliographystyle{amsplain}

\begin{thebibliography}{10}

\bibitem{ABR13}
A.~A. Albanese, J.~Bonet, and W.~J. Ricker, \emph{Montel resolvents and
  uniformly mean ergodic semigroups of linear operators}, Quaest. Math.
  \textbf{36} (2013), no.~2, 253--290.

\bibitem{Allan}
G.~R. Allan, \emph{A spectral theory for locally convex alebras}, Proc. London
  Math. Soc. (3) \textbf{15} (1965), 399--421.

\bibitem{ARZ}
H.~Arikan, L.~Runov, and V.~Zahariuta, \emph{Holomorphic functional calculus
  for operators on a locally convex space}, Results Math. \textbf{43} (2003),
  no.~1-2, 23--36.

\bibitem{EN}
K.-J. Engel and R.~Nagel, \emph{One-parameter semigroups for linear evolution
  equations}, Springer-Verlag, New York, 2000.

\bibitem{FEJKW}
L.~Frerick, E.~Jord\'{a}, T.~Kalmes, and J.~Wengenroth, \emph{Strongly
  continuous semigroups on some {F}r\'{e}chet spaces}, J. Math. Anal. Appl.
  \textbf{412} (2014), no.~1, 121--124.

\bibitem{Jarchow}
H.~Jarchow, \emph{Locally {C}onvex {S}paces}, B. G. Teubner, Stuttgart, 1981.

\bibitem{J}
G.~A. Joseph, \emph{Boundedness and completeness in locally convex spaces and
  algebras}, J. Austral. Math. Soc. Ser. A \textbf{24} (1977), no.~1, 50--63.

\bibitem{Komura}
T.~K{\=o}mura, \emph{Semigroups of operators in locally convex spaces}, J.
  Functional Analysis \textbf{2} (1968), 258--296.

\bibitem{KoetheII}
G.~K{\"o}the, \emph{Topological vector spaces. {I}, {II}}, Springer-Verlag, New
  York, 1969/79.

\bibitem{Moore}
R.~T. Moore, \emph{Banach algebras of operators on locally convex spaces},
  Bull. Amer. Math. Soc. \textbf{75} (1969), 68--73.

\bibitem{Yosida}
K.~Yosida, \emph{Functional analysis}, Springer-Verlag, Berlin, 1995, Reprint
  of the sixth (1980) edition.

\end{thebibliography}

\end{document}